\documentclass[twoside,leqno,10pt, A4]{amsart}
\usepackage{amsfonts}
\usepackage{amsmath}
\usepackage{amscd}
\usepackage{amssymb}
\usepackage{amsthm}
\usepackage{amsrefs}
\usepackage{latexsym}
\usepackage{mathrsfs}
\usepackage{bbm}
\usepackage{amscd}
\usepackage{amssymb}
\usepackage{amsthm}
\usepackage{amsrefs}
\usepackage{latexsym}
\usepackage{mathrsfs}
\usepackage{bbm}
\usepackage{enumerate}
\usepackage{graphicx}
\usepackage{color}
\begin{document}

\newtheorem{theorem}[subsection]{Theorem}
\newtheorem{proposition}[subsection]{Proposition}
\newtheorem{lemma}[subsection]{Lemma}
\newtheorem{corollary}[subsection]{Corollary}
\newtheorem{conjecture}[subsection]{Conjecture}
\newtheorem{prop}[subsection]{Proposition}
\newtheorem{defin}[subsection]{Definition}

\numberwithin{equation}{section}
\newcommand{\mr}{\ensuremath{\mathbb R}}
\newcommand{\mc}{\ensuremath{\mathbb C}}
\newcommand{\dif}{\mathrm{d}}
\newcommand{\intz}{\mathbb{Z}}
\newcommand{\ratq}{\mathbb{Q}}
\newcommand{\natn}{\mathbb{N}}
\newcommand{\comc}{\mathbb{C}}
\newcommand{\rear}{\mathbb{R}}
\newcommand{\prip}{\mathbb{P}}
\newcommand{\uph}{\mathbb{H}}
\newcommand{\fief}{\mathbb{F}}
\newcommand{\majorarc}{\mathfrak{M}}
\newcommand{\minorarc}{\mathfrak{m}}
\newcommand{\sings}{\mathfrak{S}}
\newcommand{\fA}{\ensuremath{\mathfrak A}}
\newcommand{\mn}{\ensuremath{\mathbb N}}
\newcommand{\mq}{\ensuremath{\mathbb Q}}
\newcommand{\half}{\tfrac{1}{2}}
\newcommand{\f}{f\times \chi}
\newcommand{\summ}{\mathop{{\sum}^{\star}}}
\newcommand{\chiq}{\chi \bmod q}
\newcommand{\chidb}{\chi \bmod db}
\newcommand{\chid}{\chi \bmod d}
\newcommand{\sym}{\text{sym}^2}
\newcommand{\hhalf}{\tfrac{1}{2}}
\newcommand{\sumstar}{\sideset{}{^*}\sum}
\newcommand{\sumprime}{\sideset{}{'}\sum}
\newcommand{\sumprimeprime}{\sideset{}{''}\sum}
\newcommand{\sumflat}{\sideset{}{^\flat}\sum}
\newcommand{\shortmod}{\ensuremath{\negthickspace \negthickspace \negthickspace \pmod}}
\newcommand{\V}{V\left(\frac{nm}{q^2}\right)}
\newcommand{\sumi}{\mathop{{\sum}^{\dagger}}}
\newcommand{\mz}{\ensuremath{\mathbb Z}}
\newcommand{\leg}[2]{\left(\frac{#1}{#2}\right)}
\newcommand{\muK}{\mu_{\omega}}
\newcommand{\thalf}{\tfrac12}
\newcommand{\lp}{\left(}
\newcommand{\rp}{\right)}
\newcommand{\Lam}{\Lambda_{[i]}}
\newcommand{\lam}{\lambda}
\newcommand{\af}{\mathfrak{a}}
\newcommand{\sw}{S_{[i]}(X,Y;\Phi,\Psi)}
\newcommand{\lz}{\left(}
\newcommand{\pz}{\right)}
\newcommand{\bfrac}[2]{\lz\frac{#1}{#2}\pz}
\newcommand{\odd}{\mathrm{\ primary}}
\newcommand{\even}{\text{ even}}
\newcommand{\res}{\mathrm{Res}}
\newcommand{\sumn}{\sumstar_{(c,1+i)=1}  w\left( \frac {N(c)}X \right)}
\newcommand{\lab}{\left|}
\newcommand{\rab}{\right|}
\newcommand{\Go}{\Gamma_{o}}
\newcommand{\Ge}{\Gamma_{e}}
\newcommand{\M}{\widehat}

\theoremstyle{plain}
\newtheorem{conj}{Conjecture}
\newtheorem{remark}[subsection]{Remark}

\makeatletter
\def\widebreve{\mathpalette\wide@breve}
\def\wide@breve#1#2{\sbox\z@{$#1#2$}%
     \mathop{\vbox{\m@th\ialign{##\crcr
\kern0.08em\brevefill#1{0.8\wd\z@}\crcr\noalign{\nointerlineskip}%
                    $\hss#1#2\hss$\crcr}}}\limits}
\def\brevefill#1#2{$\m@th\sbox\tw@{$#1($}%
  \hss\resizebox{#2}{\wd\tw@}{\rotatebox[origin=c]{90}{\upshape(}}\hss$}
\makeatletter

\title[Negative first moment of quadratic twists of $L$-functions]{Negative first moment of quadratic twists of $L$-functions}

%%\date{\today}
\author[P. Gao]{Peng Gao}
\address{School of Mathematical Sciences, Beihang University, Beijing 100191, China}
\email{penggao@buaa.edu.cn}

\author[L. Zhao]{Liangyi Zhao}
\address{School of Mathematics and Statistics, University of New South Wales, Sydney NSW 2052, Australia}
\email{l.zhao@unsw.edu.au}

\begin{abstract}
 We evaluate asymptotically the negative first moment at points larger than $1/2$ of the family of quadratic twists of automorphic $L$-functions using multiple Dirichlet series  under the generalized Riemann hypothesis and the Ramanujan-Petersson conjecture.
\end{abstract}

\maketitle

\noindent {\bf Mathematics Subject Classification (2010)}: 11M06, 11M41, 11L05, 11F66  \newline

\noindent {\bf Keywords}:  negative moment, automorphic $L$-functions, multiple Dirichlet series, quadratic twists

\section{Introduction}\label{sec 1}

Moments of central values of $L$-functions have been studied extensively for their rather salient utility in tackling a lot of number theoretic problems.  We instance the investigation of non-vanishing of $L$-functions at these central values, a subject with deep arithmetic meanings.  Although results for positive moments abound, much less is known on negative moments, even conjecturally.   For the case of the Riemann zeta function $\zeta(s)$, a conjecture of S. M. Gonek \cite{Gonek89} predicts the order of magnitude of negative moments of $\zeta(s)$ near the critical line.  Further computations on the random matrix theory side due to M. V. Berry and J. P. Keating \cites{BK02}, as well as P. J. Forrester and J. P. Keating \cite{FK}, suggest extra transition regimes in Gonek's conjecture.  Note that sharp lower bounds for the negative moments of $\zeta(s)$ are also obtained in \cite{Gonek89} for certain ranges under the Riemann hypothesis (RH). \newline

 For the family of quadratic Dirichlet $L$–functions, almost sharp upper bounds are obtained in \cite{BFK21} for the negative moments in the function fields setting at points slightly shifted away from the central point. In \cite{Florea21}, asymptotic formulas are achieved for these negative moments for certain ranges of the shifts over function fields.  Additionally, lower bounds are given in \cite{Gao2022} for negative moments of families of quadratic Dirichlet $L$-functions at the central point  in
the number field setting, assuming the truth of a conjecture of S. Chowla \cite{chow} on the non-vanishing of these $L$-values. Note that the results in \cite{Gao2022} are only expected to be sharp for the $2k$-th moment with $-5/2 \leq k<0$ since the work in \cite{FK} also suggests certain phase changes in the asymptotic formulas for the $2k$-th moment of the family of quadratic Dirichlet $L$-functions when $2k = -(2j + 1/2)$ for any positive integer $j$.  We also point out here that asymptotic
formulas for negative moments of the above family of $L$-functions at $1$ were obtained by A. Granville and K. Soundararajan \cite{GS03} in the number field setting and by A. Lumley \cite{Lumley} for function fields. \newline

Our aim in this paper is to evaluate asymptotically the negative first moment of the family of quadratic twists of automorphic $L$-functions
at $1/2+\alpha$ with $0 < \alpha \leq 1/2$.  We remark here first that we need to assume the truth of the generalized Riemann hypothesis (GRH) so that the
$L$-functions under consideration do not vanish those points.  Also, since GRH alone does not rule out central vanishing of $L$-functions, it is appropriate to consider negative moments of these $L$-functions at points that are away from the central point. \newline

  To state our result, we recall that according to the Langlands program (see \cite{Langlands}), the most general $L$-function is that attached to an
automorphic representation of $\text{GL}_N$ over a number field, which in turn can be written as products of the $L$-functions attached to cuspidal automorphic representations of $\text{GL}_M$ over the $\mq$. \newline

  We fix a self-contragredient representation $\pi$ of $\text{GL}_{M}$ over $\mq$ so that $\pi = \tilde{\pi}$ and we write $L(s, \pi)$ for the $L$-function attached to $\pi$.  For sufficiently large $\Re (s)$, $L(s, \pi)$ can be expressed as an Euler product of the form
\begin{align*}
%%\label{Lpi}
 L(s, \pi)=\prod_pL(s, \pi_p)=\prod_p\prod^M_{j=1}(1-\alpha_{\pi}(p, j)p^{-s})^{-1}.
\end{align*}
Here and throughout, we reserve the letter $p$ for a prime number and $\varepsilon$ for an arbitrarily small positive number that may not be the same at each occurrence.  Moreover, all implied constants may depend on $\varepsilon$.  \newline

We also assume the truth of the Ramanujan-Petersson conjecture, which implies that
\begin{equation} \label{Ramanujanconj}
   |\alpha_{\pi}(p, j)| \leq 1.
\end{equation}
The Rankin-Selberg symmetric square $L$-function $L(s, \pi \otimes \pi)$ factors as the product of the symmetric and exterior square $L$-functions
(see \cite[p. 139]{BG}):
\begin{align*}
%%\label{RankinSelbergdecomp}
 L(s, \pi \otimes \pi)=L(s, \vee^2)L(s, \wedge^2),
\end{align*}
   and has a simple pole at $s = 1$  which is carried by one of the two factors. We write the
order of the pole of $L(s, \wedge^2)$ as $(\delta(\pi) + 1)/2$  so that $\delta(\pi) = \pm 1$.
The order of the pole of $L(s, \vee^2)$ then equals to $(1-\delta(\pi) )/2$. We further choose $\pi$ so that $\delta(\pi)=-1$. \newline

  Let $\chi$ be any Dirichlet character and write $\chi_m$ for the quadratic character $\left(\frac {\cdot}{m} \right)$ for any odd,
positive integer $m$. We further denote $L(s, \pi \otimes \chi)$ for the $L$-function twisted by $\chi$. Note that when $\Re(s)>1$, $L(s, \pi
\otimes \chi)$ has an Euler product given by
\begin{align*}
%%\label{Lpitwist}
 L(s, \pi \otimes \chi)=\prod_p\prod^M_{j=1}(1-\chi(p)\alpha_{\pi}(p, j)p^{-s})^{-1}.
\end{align*}

  We then obtain that for $\Re(s)>1$,
\begin{align}
\begin{split}
\label{Linverse}
 L^{-1}(s, \pi \otimes \chi)=\prod_p\prod^M_{j=1}(1-\alpha_{\pi}(p, j)\chi(p)p^{-s})=:\sum^{\infty}_{n=1}\frac {a_{\pi}(n)\chi(n)}{n^s}.
\end{split}
\end{align}

  Note that it follows from \eqref{asumestimation} below and partial summation that the series given in \eqref{Linverse} converges absolutely
for $\Re(s)>1/2+\varepsilon$ under GRH.  In particular, $L^{-1}(s, \pi \otimes \chi)$ can be meromorphically continued to the region $\Re(s)>1/2$. \newline

    For any $L$-function, $L^{(c)}$ (resp. $L_{(c)}$) denotes the function given by the Euler product defining $L$ but omitting
those primes dividing (resp. not dividing) $c$. We also write $L_p$ for $L_{(p)}$.   Our main result on the negative first moment of the quadratic twists of automorphic $L$-functions is as follows.
\begin{theorem}
\label{Theorem for all characters}
		Suppose that GRH and the Ramanujan-Petersson conjecture are true. Let $M \geq 1$ and let $\pi$ be a fixed self-contragredient representation
of $\text{GL}_{M}$ over $\mq$ such that $\delta(\pi)=-1$.  Let $w(t)$ be a non-negative Schwartz function compactly supported on $(0, \infty)$ and $\widehat w(s)$
be its Mellin transform.  For $0<\alpha \leq 1/2$, we have, for any $\varepsilon>0$,
\begin{align}
\label{Asymptotic for ratios of all characters}
\begin{split}	
&\sum_{\substack{(n,2)=1}}\frac{1}{L^{(2)}(\tfrac{1}{2}+\alpha, \pi \otimes \chi_{n})}w \Big( \frac nX \Big) =  X\M w(1)P(\tfrac 12+\alpha; \pi)
+O((1+|\alpha|)^{1+\varepsilon}X^{1-2\alpha+\varepsilon}),
\end{split}
\end{align}
 where
\begin{equation}
\label{Pz}
		P(z;\pi)=\frac12\prod_{p>2}\lz1+\lz 1-\frac 1p \pz \sum^{\infty}_{k=1}\frac {a_{\pi}(p^{2k})}{p^{2kz}}\pz .
\end{equation} 	
\end{theorem}
	
The function $P(z;\pi)$ converges absolutely when $\Re(z) >1/2+\varepsilon$ in view of \eqref{abound} below. We also note that our proof of Theorem \ref{Theorem for all characters} in fact establishes an asymptotic formula for the left-hand side of \eqref{Asymptotic for ratios of all characters} for all $\alpha>0$. Our proof of Theorem \ref{Theorem for all characters} is motivated by the work of M. \v Cech \cite{Cech1} on $L$-functions ratios conjectures for the family of quadratic Dirichlet $L$-functions over $\mq$ using the method of multiple Dirichlet series. \newline

In general, one may apply the method of multiple Dirichlet series to study anything expressable in terms of multivariable integrals involving with a multiple Dirichlet series.  One then treats the multiple Dirichlet series as a function of several complex variables and the main task is to establish analytical properties of the multiple Dirichlet series and obtain meromorphic continuations of the series to a region as large as possible. Often this is achievable by developing enough functional equations.  One may then use techniques of contour integrals. \newline

  In \cite{Cech1}, M. \v Cech applied a Poisson-type summation formula to obtain the functional equation for a Dirichlet $L$-function, generalizing the known functional equation for Dirichlet $L$-functions attached to primitive characters.   This type of Poisson summation was initially developed by K. Soundararajan in \cite{sound1} to study the non-vanishing of the central values of quadratic Dirichlet $L$-functions.  Note that in \eqref{Asymptotic for ratios of all characters}, we consider the average of $(L^{(2)}(\tfrac{1}{2}+\alpha, \pi \otimes \chi_{n}))^{-1}$ over quadratic twists of automorphic functions by general Dirichlet characters instead of just primitive ones. This has the advantage that our proof of Theorem \ref{Theorem for all characters} may utilize the functional equations for general Dirichlet $L$-functions in \cite{Cech1} to achieve a better error term. \newline

Throughout the paper, we use ``$\pi$" for both a fixed self-contragredient representation of $\text{GL}_{M}$ over $\mq$ and the transcendental number $3.1415 \cdots$, as they are both standard notations.  This will nevertheless cause no ambiguity as the symbol's meaning will be clear from the context.

\section{Preliminaries} \label{sec 2}

In this section, we include some auxiliary results needed in the paper.

%%----------------------------------------------------------------------------
\subsection{Quadratic Gauss sums}
\label{sec2.4}
%%----------------------------------------------------------------------------
    Following the notation given in \cite[Chapter 3]{iwakow}, we denote by $\chi^{(m)}$ the Kronecker symbol
$\left(\frac {m}{\cdot} \right)$ for any integer $m \equiv 0, 1 \pmod 4$.
We further denote $\psi_j, j \in \{ \pm 1, \pm 2\}$ be quadratic characters given
by $\psi_j=\chi^{(4j)}$. Note that $\psi_j$ is a primitive character modulo $4j$ for each $j>1$.
We also denote $\psi_0$ the primitive principal character. \newline

  For any integer $q$ and any Dirichlet character $\chi$ modulo $n$, we define the associated Gauss sum $\tau(\chi,q)$ by
\begin{equation*}
%%\label{key}
		\tau(\chi,q)=\sum_{j\bmod n}\chi(j)e \left( \frac {jq}n \right), \quad \mbox{where} \quad e(z) = \exp (2 \pi i z).
\end{equation*}
We quote the following result from \cite[Lemma 2.2]{Cech1}.
\begin{lemma}
\label{Lemma changing Gauss sums}
\begin{enumerate}
\item If $l\equiv1 \pmod 4$, then
\begin{equation*}
%%\label{key}
		\tau (\chi^{(4l)},q )=
\begin{cases}
					0,&\hbox{if $(q,2)=1$,}\\
					-2\tau\lz\chi_l,q\pz,&\hbox{if $q\equiv2 \pmod 4$,}\\
					2\tau\lz \chi_l ,q\pz,&\hbox{if $q\equiv0 \pmod 4$.}
\end{cases}
\end{equation*}
			\item If $l \equiv3 \pmod 4$, then
\begin{equation*}
%%\label{key}
				\tau(\chi^{(4l)},q)=\begin{cases}
					0,&\hbox{if $2|q$,}\\
					-2i\tau\lz\chi_l,q\pz,&\hbox{if $q\equiv1 \pmod 4$,}\\
					2i\tau\lz\chi_l,q\pz,&\hbox{if $q\equiv3 \pmod 4$.}
				\end{cases}
\end{equation*}
		\end{enumerate}
\end{lemma}
	
	For any odd, positive integer $n$, we further define $G\lz\chi_n,q\pz$ by
\begin{align}
\label{key}
\begin{split}
			G\lz\chi_n,q\pz&=\lz\frac{1-i}{2}+\leg{-1}{n}\frac{1+i}{2}\pz\tau\lz\chi_n,q\pz=\begin{cases}
				\tau\lz\chi_n,q\pz,&\hbox{if $n\equiv1 \pmod 4$,}\\
				-i\tau\lz\chi_n,q\pz,&\hbox{if $n\equiv3\pmod 4$}.
			\end{cases}
\end{split}
\end{align}
	
Let $\varphi(m)$ denote the Euler totient function of $m$.  Our next result is taken from \cite[Lemma 2.3]{sound1} that evaluates $G\lz\chi_m,q\pz$.
\begin{lemma}
\label{lem:Gauss}
   If $(m,n)=1$ then $G(\chi_{mn},q)=G(\chi_m,q)G(\chi_n,q)$. Suppose that $p^a$ is
   the largest power of $p$ dividing $q$ (put $a=\infty$ if $q=0$).
   Then for $k \geq 0$ we have
\begin{equation*}
%%\label{Sound's Gauss sums - exact formula}
		G\lz\chi_{p^k},q\pz=\begin{cases}\varphi(p^k),&\hbox{if $k\leq a$, $k$ even,}\\
			0,&\hbox{if $k\leq a$, $k$ odd,}\\
			-p^a,&\hbox{if $k=a+1$, $k$ even,}\\
			\leg{qp^{-a}}{p}p^{a}\sqrt p,&\hbox{if $k=a+1$, $k$ odd,}\\
			0,&\hbox{if $k\geq a+2$}.
		\end{cases}
\end{equation*}
\end{lemma}
	
\subsection{Automorphic $L$-functions}
\label{sec: automorphic}

    Recall that $\pi$ is a fixed self-contragredient representation of $\text{GL}_{M}$ over $\mq$ and that $L(s, \pi)$ is the associated $L$-function.
For any Dirichlet character $\chi$, recall from \eqref{Linverse} that when $\Re(s)>1$, we have
\begin{align*}
\begin{split}
%%\label{Linverse1}
 L^{-1}(s, \pi \otimes \chi)=\sum^{\infty}_{n=1}\frac {a_{\pi}(n)\chi(n)}{n^s}.
\end{split}
\end{align*}
  Here we notice that the Euler product given in \eqref{Linverse} implies that $a_{\pi}(n)$ is a multiplicative function of $n$ such that
\begin{align}
\begin{split}
\label{aexp}
  a_{\pi}(p^k)=
\displaystyle
\begin{cases}
 (-1)^k \displaystyle \sum_{\substack{1 \leq j_1, \cdots, j_k \leq M \\ j_i \text{distinct}}}\prod^{k}_{i=1} \alpha_{\pi}(p, j_i), \quad 0 \leq k \leq M, \\
   0, \quad  k>M,
\end{cases}
\end{split}
\end{align}
where the empty product, as usual, is defined to be $1$. \newline

In particular, under the Ramanujan-Petersson conjecture, we have $|a_{\pi}(p^k)| \leq 2^M$ for all $k \geq 0$, so
\begin{align}
\begin{split}
\label{abound}
  a_{\pi}(n) \ll (2^M)^{\omega(n)} \ll n^{\varepsilon},
\end{split}
\end{align}
  where $\omega(n)$ is the number of distinct primes dividing $n$ and the last estimate above follows from the well-known bound (see \cite[Theorem 2.10]{MVa})
\begin{align}
\label{omegabound}
   \omega(h) \ll \frac {\log h}{\log \log h}, \quad \mbox{for} \quad h \geq 3.
\end{align}

   For a fundamental discriminant $d$ ($d$ is square-free, $d \equiv 1 \pmod 4$ or when $d=4m$ with $m$ square-free, $m \equiv 2,3 \pmod 4$), the character $\chi^{(d)}$ is primitive modulo $|d|$ and the function $L(s, \pi \otimes \chi^{(d)})$
 has an analytical continuation to the entirety of $\comc$ and has a functional equation of the form (see \cite[Section 3.6]{Ru})
\begin{align*}
%%\label{Lambdafcn}
 \Lambda(s,  \pi \otimes \chi^{(d)}): =&\pi^{-Ms/2}\prod^M_{j=1}\Gamma \left( \frac {s+\mu_{\pi \otimes \chi^{(d)}}(j)}{2}\right) L(s, \pi \otimes
\chi^{(d)}) = \epsilon(s, \pi \otimes \chi^{(d)})\Lambda(1-s,  \pi \otimes \chi^{(d)}),
\end{align*}
  where $\mu_{\pi \otimes \chi^{(d)}}(j) \in \mc$ and satisfies $\Re(\mu_{\pi \otimes \chi_{d}}(j)) \geq 0$ since we are assuming GRH. \newline

   Moreover, as $\delta(\pi) =-1$, we have
\begin{align*}
%%\label{epsilons}
 \epsilon(s, \pi \otimes \chi^{(d)})=Q^{-s+1/2}_{\pi \otimes \chi^{(d)}},
\end{align*}
  where $Q_{\pi \otimes \chi^{(d)}}$ is the conductor of $L(s, \pi \otimes \chi^{(d)})$ and it is known (see \cite[Section 2]{M&T-B}) that
\begin{align}
\label{Qbound}
  Q_{\pi \otimes \chi^{(d)}} \ll d^M.
\end{align}

\subsection{Bounds on $L$-functions}

   We now gather various estimates on the values of $L$-functions of our interest.
For any quadratic Dirichlet character $\chi$ modulo $n$, we write $\widehat{\chi}$ for the primitive character that induces $\chi$. It follows from \cite[Theroem 9.13]{MVa} that $\widehat{\chi}=\chi^{(d)}$ for some fundamental discriminant $d$.  The conductor of $\widehat{\chi}$ then equals $d$, which divides $n$. We further note that if $d$ is a fundamental discriminant, then one has the following functional equation (see \cite[p. 456]{sound1}) for $L(s, \chi^{(d)})$.
\begin{align} \label{fneqnquad}
  \Lambda(s, \chi^{(d)}) =: \left(\frac {|d|} {\pi} \right)^{s/2}\Gamma \Big( \frac {s}{2} \Big)L(s, \chi^{(d)})=\Lambda(1-s,  \chi^{(d)}).
\end{align}
It follows from \eqref{fneqnquad} and \cite[Theorem 5.19, Corollary 5.20]{iwakow} that, under GRH, we have for any fundamental discriminant $d$ and $\Re(s) \geq 1/2$,
\begin{align}
\label{Lchiupperbound}
\begin{split}
&  \big| L(s, \chi^{(d)}) \big |  \ll |sd|^{\varepsilon}.
\end{split}
\end{align}
Stirling's formula (see \cite[(5.113)]{iwakow}) yields
\begin{align}
\label{Stirlingratio1}
  \frac {\Gamma(\frac {1-s}{2})}{\Gamma (\frac s2)} \ll (1+|s|)^{1/2-\Re (s)}.
\end{align}

  We apply \eqref{fneqnquad} and \eqref{Stirlingratio1}, together with the convexity bound for $L(s, \chi^{(d)})$ (see \cite[Exercise 3, p. 100]{iwakow}) to get that for any fundamental discriminant $d$,
\begin{align} \label{Lchidbound}
\begin{split}
   L(s, \chi^{(d)}) \ll \begin{cases}
   1 \qquad & \Re(s) >1,\\
   (|d|(1+|s|))^{(1-\Re(s))/2+\varepsilon} \qquad & 0< \Re(s) <1,\\
    (|d|(1+|s|))^{1/2-\Re(s)+\varepsilon} \qquad & \Re(s) \leq 0.
\end{cases}
\end{split}
\end{align}
Next, we note that
\begin{align} \label{Ldecomp}
\begin{split}
 L\lz s, \pi \otimes \chi \pz =L\lz s, \pi \otimes \widehat{\chi} \pz \prod_{p|n}\prod^M_{j=1} \left( 1-\frac {\alpha_{\pi}(p,j)\widehat\chi(p)}{p^s} \right).
\end{split}
\end{align}

Applying \eqref{Ramanujanconj} gives that if $\Re(s) \geq 1/2$, then for $1 \leq j \leq M$,
\begin{align*}
%%\label{key}
 \Big |1-\frac {\alpha_{\pi}(p,j)\widehat\chi(p)}{p^s}\Big |^{-1} \leq (1-2^{-1/2})^{-1}.
\end{align*}

  It follows from the above and \eqref{omegabound} that, for $\Re(s) \geq 1/2$,
\begin{align}
\label{Lninvbound}
\begin{split}
 \prod_{p|n}\prod^M_{j=1}\Big |1-\frac {\alpha_{\pi}(p,j)\widehat\chi(p)}{p^s}\Big |^{-1} \ll ((1-2^{-1/2})^{-M})^{\omega(n)} \ll n^{\varepsilon}.
\end{split}
\end{align}

  Now, we apply  \eqref{Qbound} and \cite[Theorem 5.19]{iwakow} to infer that, under GRH, we have, for $\Re(s) \geq1/2+\varepsilon$,
\begin{align}
\label{Linvupperbound}
\begin{split}
&  \big| L\lz s, \pi \otimes \widehat{\chi} \pz \big |^{-1}  \ll |sn|^{\varepsilon}.
\end{split}
\end{align}

Summarizing \eqref{Ldecomp}--\eqref{Linvupperbound}, we get that for $\Re(s) \geq 1/2+\varepsilon$,
\begin{align}
\label{Linvgenupperbound}
\begin{split}
 \big| L\lz s, \pi \otimes \chi \pz \big |^{-1}  \ll |sn|^{\varepsilon}.
\end{split}
\end{align}

\subsection{Functional equations for quadratic Dirichlet $L$-functions}
	
	A key ingredient needed in our proof of Theorem \ref{Theorem for all characters} is the following functional equation for any Dirichlet character $\chi$ modulo $n$ from \cite[Proposition 2.3]{Cech1}.
\begin{lemma}
\label{Functional equation with Gauss sums}
		Let $\chi$ be any Dirichlet character modulo $n \neq \square$ such that $\chi(-1)=1$. Then we have
\begin{equation}
\label{Equation functional equation with Gauss sums}
			L(s,\chi)=\frac{\pi^{s-1/2}}{n^s}\frac{\Gamma\bfrac{1-s}{2}}{\Gamma\bfrac {s}2} K(1-s,\chi), \quad \mbox{where} \quad K(s,\chi)=\sum_{q=1}^\infty\frac{\tau(\chi,q)}{q^s} .
\end{equation}
\end{lemma}

\subsection{Some results on multivariable complex functions}
	
   We include in this section some results from multivariable complex analysis. First we need the notation of a tube domain.
\begin{defin}
		An open set $T\subset\mc^n$ is a tube if there is an open set $U\subset\mr^n$ such that $T=\{z\in\mc^n:\ \Re(z)\in U\}.$
\end{defin}
	
   For a set $U\subset\mr^n$, we define $T(U)=U+i\mr^n\subset \mc^n$.  We have the following Bochner's Tube Theorem \cite{Boc}.
\begin{prop}
\label{Bochner}
		Let $U\subset\mr^n$ be a connected open set and $f(z)$ be a function that is holomorphic on $T(U)$. Then $f(z)$ has a holomorphic continuation to the convex hull of $T(U)$.
\end{prop}

The convex hull of an open set $T\subset\mc^n$ is denoted by $\hat T$.  Then we quote the result from \cite[Proposition C.5]{Cech1} concerning the modulus of holomorphic continuations of functions in multiple variables.
\begin{prop}
\label{Extending inequalities}
		Assume that $T\subset \mc^n$ is a tube domain, $g,h:T\rightarrow \mc$ are holomorphic functions, and let $\tilde g,\tilde h$ be their holomorphic continuations to $\hat T$. If  $|g(z)|\leq |h(z)|$ for all $z\in T$, and $h(z)$ is nonzero in $T$, then also $|\tilde g(z)|\leq |\tilde h(z)|$ for all $z\in \hat T$.
\end{prop}

\section{Proof of Theorem \ref{Theorem for all characters}}

  Recall from \eqref{Linverse} that for any Dirichlet character $\chi$ modulo $n$, we have, for $\Re(s)>1$,
\begin{align}
\begin{split}
\label{Linv}
 L^{-1}(s, \pi \otimes \chi)=\sum^{\infty}_{m=1}\frac {a_{\pi}(m)\chi(m)}{m^s}.
\end{split}
\end{align}
  We next deduce from \eqref{abound} and Perron's formula given as in Theorem 5.2 and Corollary 5.3 in \cite{MVa} that with
$\sigma_0 = 1 + 1/ \log x$ and $1 \leq  T \leq x$ for any $x \geq 2$,
\begin{align*}
\begin{split}
%%\label{ansum}
  \sum_{m \leq x}a_{\pi}(m)\chi(m)=\frac 1{2\pi i}\int\limits^{\sigma_0+iT}_{\sigma_0-iT}\frac {x^s}{L(s, \pi \otimes \chi)} \frac{\dif s}{s} +O\left( \frac {x^{1+\varepsilon}}{T} \right).
\end{split}
\end{align*}

  Upon shifting the line of integration above to $\Re(s)=1/2+\varepsilon$, applying \eqref{Linvgenupperbound} to estimate $L(s, \pi \otimes \chi)^{-1}$, setting $T=x^{1/2}$ and invoking GRH, we get
\begin{align}
\begin{split}
\label{asumestimation}
  \sum_{m \leq x}a_{\pi}(m)\chi(m) \ll n^{\varepsilon} x^{1/2+\varepsilon}.
\end{split}
\end{align}

  It follows that the series in \eqref{Linv} has holomorphic continues to the region $\Re(s)>1/2+\varepsilon$ under GRH. This allows us to define for $\Re(s)$ large enough and $\Re(z)>1/2$,
\begin{align}
\label{Aswzexp}
\begin{split}
A(s, z;\pi)=& \sum_{\substack{(n,2)=1 }}\frac 1{L^{(2)}(z,\pi \otimes \chi_{n})n^s}
=\sum_{\substack{(nk,2)=1}}\frac{a_{\pi}(k)\chi_n(k)}{k^zn^s}= \sum_{\substack{(k,2)=1}}\frac{a_{\pi}(k)L(s,\chi^{(4k)} )}{k^z}.
\end{split}
\end{align}

 We shall devote the next few sections to the study of the analytic properties of $A(s, z; \pi)$, which are crucial in the proof of Theorem \ref{Theorem for all characters}.
	
\subsection{First region of absolute convergence of $A(s,z;\pi)$}

    Note that quadratic reciprocity law implies that if $n \equiv 1 \pmod 4$, we have $L^{(2)}(s,\pi \otimes \chi_n)=L(s,\pi \otimes \chi^{(4n)})$, while $n \equiv -1 \pmod 4$ gives $L^{(2)}(s,\pi \otimes \chi_n)=L(s, \pi \otimes \chi^{(-4n)})$.  It follows from this and the series representation for $A(s,z;\pi)$ given by the first equality in \eqref{Aswzexp} that, for $\Re(z)\geq 1/2+\varepsilon$,
\begin{align}
\begin{split}
\label{Abound}
		A(s,z;\pi)=& \sum_{\substack{(n,2)=1}}\frac{1}{L^{(2)}(z, \pi \otimes \chi_n)n^s} = \sum_{\substack{n \equiv 1 \shortmod 4}}\frac{1}{L(z, \pi \otimes \chi^{(4n)})n^s}+\sum_{\substack{n \equiv -1 \shortmod 4}}\frac{1}{L(z, \pi \otimes \chi^{(-4n)})n^s}\\
\ll&  |z|^{\varepsilon}\sum_{\substack{n}}\frac{1}{n^{s-\varepsilon}},
\end{split}
\end{align}
  where the last bound is a consequence of \eqref{Linvgenupperbound}.  It follows from the above that $A(s,z;\pi)$ converges absolutely in the region
\begin{equation*}
%%\label{key}
		S_0=\{(s,z): \Re(s)>1, \ \Re(z)> \tfrac{1}{2} \}.
\end{equation*}

  We also deduce from the last expression of \eqref{Aswzexp} that $A(s,z;\pi)$ is given by the series
\begin{align}
\label{Sum A(s,w,z) over n}		
A(s,z;\pi)=&\sum_{\substack{(k,2)=1}}\frac{a_{\pi}(k) L( s, \chi^{(4k)})}{k^z}.
\end{align}

We write any odd integer $k$ uniquely as $k=k_0k^2_1$ with $k_0$ odd and square-free. Then
\begin{align}
\label{Lsmkbound}	
\begin{split}	
L( s, \chi^{(4k)}) \ll & (4k)^{\max (0,-\Re(s))+\varepsilon} |L( s, \widehat\chi^{(4k)})|,
\end{split}
\end{align}
 where we recall that $\widehat\chi^{(4k)}$ denotes the primitive Dirichlet character that induces $\chi^{(4k)}$. \newline

 We now apply \eqref{Lchiupperbound} and \eqref{Lsmkbound} for the case $\Re(s) \geq 1/2$, together with the functional equation in \eqref{fneqnquad} for $L( s, \widehat\chi^{(4k)})$ with $\Re(s)<1/2$ to \eqref{Sum A(s,w,z) over n} to see that, except for a simple pole at $s=1$ arising from the summands with $k=\square$, the sum over $k$ in \eqref{Sum A(s,w,z) over n} converges absolutely in the region
\begin{align*}
%%\label{key}
		S_1=& \{(s,z):\hbox{$\Re(z)>1,\ \Re(s)\geq \frac 12$}\} \bigcup \{(s,z):\hbox{$0 \leq \Re(s)< \frac 12, \ \Re(s+z)>\frac 32$}\}\bigcup \{(s, z):\hbox{$\Re(s)<0,  \ \Re(2s+z)>\frac 32 $}\}.
\end{align*}

	The convex hull of $S_0$ and $S_1$ is
\begin{equation}
\label{Region of convergence of A(s,w,z)}
		S_2=\{(s,z):\Re(z)> \tfrac{1}{2}, \Re(s+z)> \tfrac{3}{2}, \ \Re(2s+z)>\tfrac32\}.
\end{equation}

   We thus conclude by our discussions above and Proposition \ref{Bochner} that $(s-1)A(s, z;\pi)$ has meromorphic continuation to the region $S_2$.

\subsection{Residue of $A(s,z;\pi)$ at $s=1$}
	
	It follows from \eqref{Sum A(s,w,z) over n} that $A(s,z;\pi)$ has a pole at $s=1$ arising from the terms with $k=\square$. In this case, we have
\begin{equation*}
		L( s, \chi^{(4k)})=\zeta(s)\prod_{p|2k} \left( 1-\frac1{p^s} \right).
\end{equation*}
    Recall that the residue of $\zeta(s)$ at $s = 1$ equals $1$.  Then we have	
\begin{equation*}
%%\label{key}
 \res_{s=1}A(s,z;\pi)= \frac 12\sum_{\substack{k=\square \\ (k,2)=1 }}\frac{a_{\pi}(k)}{k^z}\prod_{p|k}(1-\frac1{p}).
\end{equation*}

   We can write the sum above as an Euler product, slightly abusing notation by writing $p^k$ for the prime factors of $k$.
We thus obtain
\begin{align}
\label{residues=1}
\begin{split}
 \res_{s=1}A(s,z;\pi)=&\frac12\prod_{p>2}\lz1+\lz 1-\frac 1p \pz \sum^{\infty}_{k=1}\frac {a_{\pi}(p^{2k})}{p^{2kz}}\pz  = P(z;\pi),
\end{split}
\end{align}
with $P(z;\pi)$ defined in \eqref{Pz}.

\subsection{Second region of absolute convergence of $A(s,z;\pi)$}

Using \eqref{Sum A(s,w,z) over n},
\begin{align} \begin{split}
\label{A1A2}
 A(s,z;\pi) =& \sum_{\substack{(k,2)=1 \\ k =  \square}}\frac{a_{\pi}(k)L( s, \chi^{(4k)})}{k^z} +\sum_{\substack{(k,2)=1 \\ k \neq \square}}\frac{a_{\pi}(k)L( s, \chi^{(4k)})}{k^z} \\
=&  \sum_{\substack{(k,2)=1 \\ k =  \square}}\frac{a_{\pi}(k)\zeta(s)\prod_{p | 2k}(1-p^{-s}) }{k^z}
+\sum_{\substack{(k,2)=1 \\ k \neq \square}}\frac{a_{\pi}(k)L( s, \chi^{(4k)}) }{k^z} =: \ A_1(s,z;\pi)+A_2(s,z;\pi), \quad \mbox{say}.
\end{split}
\end{align}

   We first note that
\begin{align*}
\begin{split}
%%\label{A1swz}
 A_1(s,z;\pi) =&  \zeta^{(2)}(s) \sum_{\substack{(k,2)=1 \\ k =  \square}}\frac {a_{\pi}(k)\prod_{p | k}(1-p^{-s})}{k^{z}}
 = \zeta^{(2)}(s) \prod_{(p, 2)=1}\Big (1+ (1-p^{-s})\sum^{\infty}_{k=1}\frac {a_{\pi}(p^k)}{p^{2kz}}\Big ).
\end{split}
\end{align*}

  It follows from the above that except for a simple pole at $s=1$, $A_1(s,z;\pi)$ is holomorphic in the region
\begin{align}
\label{S3}
		S_3=\Bigg\{(s, z):\ &\Re(s+2z)>1, \ \Re(2z)>1 \Bigg\}.
\end{align}

Now, we apply the functional equation given in Lemma \ref{Functional equation with Gauss sums} for $L(s, \chi^{(4k)})$ in the case $k \neq \square$ by observing that $\chi^{(4k)}$ is a Dirichlet character modulo $4k$ for any $k\geq1$ with $\chi^{(4k)}(-1)=1$.  Thus from \eqref{Equation functional equation with Gauss sums},
\begin{align}
\begin{split}
\label{Functional equation in s}
 A_2(s,z;\pi) =\frac{\pi^{s-1/2}}{4^s}\frac {\Gamma (\frac{1-s}2)}{\Gamma(\frac {s}2) } C(1-s, s+z;\pi),
\end{split}
\end{align}
  where $C(s,z;\pi)$ is given by the triple Dirichlet series
\begin{align*}
%%\label{key}
		C(s, z;\pi)=& \sum_{\substack{q, k \\ (k,2)=1 \\ k \neq \square}}\frac{a_{\pi}(k)\tau(\chi^{(4k)}, q)}{q^sk^z} = \sum_{\substack{q, k \\ \substack{(k,2)=1}}}\frac{a_{\pi}(k)\tau(\chi^{(4k)}, q)}{q^sk^z}-\sum_{\substack{q, k \\ (k,2)=1 \\ k = \square}}\frac{a_{\pi}(k)\tau(\chi^{(4k)}, q)}{q^sk^z}.
\end{align*}	

 By \eqref{Region of convergence of A(s,w,z)}, \eqref{S3} and the functional equation \eqref{Functional equation in s}, we see that $C(s,z;\pi)$ is initially defined in the region
\begin{equation*}
%%\label{key}
		\{(s,z):\ \Re(s+z)>\tfrac{3}{2}, \ \Re(z)> \tfrac{3}{2}, \ \Re(s+2z)> 2, \ \Re(z-s)>\tfrac 12 \}.
\end{equation*}
	To extend this region, we interchange the summations in $C(s,z;\pi)$ to obtain that
\begin{align}
\label{Cexp}
\begin{split}
  C(s, z;\pi)=& \sum^{\infty}_{q =1}\frac{1}{q^s}\sum_{\substack{(k,2)=1}}\frac{a_{\pi}(k)\tau( \chi^{(4k)}, q)}{k^z}-
  \sum^{\infty}_{q =1}\frac{1}{q^s}\sum_{\substack{(k,2)=1 \\ k = \square}}\frac{a_{\pi}(k)\tau( \chi^{(4k)}, q )}{k^z}\\
=: & \ C_1(s,z;\pi)-C_2(s,z;\pi).
\end{split}
\end{align}

   We now define for two Dirichlet characters $\psi,\psi'$ whose conductors divide $8$,
\begin{align} \begin{split}
\label{C12def}
	C_1(s,z;\psi,\psi',\pi)=: \sum_{k,q\geq 1}\frac{a_{\pi}(k)G\lz \chi_k,q\pz\psi(k)\psi'(q)}{k^zq^s} \quad \mbox{and} \quad
 C_2(s,z;\psi,\psi',\pi)=: \sum_{k,q\geq 1}\frac{a_{\pi}(k^2)G\lz \chi_{k^2},q\pz\psi(k)\psi'(q)}{k^{2z}q^s},
\end{split}
\end{align}
with $G(\chi_n, q)$ defined in \eqref{key}. \newline

  Following the arguments in \cite[\S 6.4]{Cech1} and utilizing Lemma \ref{Lemma changing Gauss sums}, we get that
\begin{align}
\begin{split}
\label{C(s,w,z) as twisted C(s,w,z)}
		C_1(s, z;\pi)=&
			-2^{-s}\big( C_1(s, z;\psi_2,\psi_1,\pi)+C_1(s,z;\psi_{-2},\psi_1,\pi)\big) +4^{-s}\big( C_1(s, z;\psi_1,\psi_0,\pi)+C_1(s,z;\psi_{-1},\psi_0,\pi)\big)\\
			& \hspace*{3cm} +C_1(s, z;\psi_1,\psi_{-1},\pi)-C_1(s, z;\psi_{-1},\psi_{-1},\pi), \\
C_2(s,z;\pi)=&
			-2^{1-s}C_2(s,w,z;\psi_1,\psi_1,\pi)+2^{1-2s}C_2(s,w,z;\psi_1,\psi_0,\pi).
\end{split}
\end{align}

Every integer $q \geq 1$ can be written uniquely as $q=q_1q^2_2$ with $q_1$ square-free. We may thus write
\begin{equation}
\label{Cidef}
		C_i(s, z;\psi,\psi',\pi)=\sumstar_{q_1}\frac{\psi'(q_1)}{q_1^s}\cdot D_i(s, z; q_1, \psi,\psi', \pi), \quad i =1,2,
\end{equation}
	where $\sum^*$ means that the sum runs over square-free integers and
\begin{align}
\label{Didef}
\begin{split}
%%\label{key}
 D_1(s, z; q_1, \psi,\psi', \pi)=&\sum_{k,q_2=1}^\infty\frac{a_{\pi}(k) G\lz \chi_{k},q_1q^2_2\pz\psi(k)\psi'(q^2_2)}{k^zq^{2s}_{2}} \quad \mbox{and} \\
 D_2(s, z; q_1, \psi,\psi', \pi)=&\sum_{k,q_2=1}^\infty\frac{a_{\pi}(k^2) G\lz \chi_{k^2},q_1q^2_2\pz\psi(k)\psi'(q^2_2)}{k^{2z}q^{2s}_{2}}.
\end{split}
\end{align}

The following result articulates the analytic properties of $D_1(s, z; q_1,  \psi,\psi', \pi)$.
\begin{lemma}
\label{Estimate For D(w,t)}
 With the notation as above and $\psi \neq \psi_0$, the functions $D_i(s, z; q_1, \psi,\psi', \pi)$ with $i=1,2$ have holomorphic continuations to the region
\begin{equation*}
%%\label{key}
		\{(s,z):\ \Re(s)>1/2,\ \Re(z)>1 \}.
\end{equation*}
Moreover, in the region $\Re(s)>1+\varepsilon$ and $\Re(z)>1+\varepsilon$, we have, for $i=1,2$,
\begin{align}
\label{Diest}
			|D_i(s, z; q_1,  \psi, \psi', \pi)|\ll |zq_1|^{\varepsilon}.
\end{align}		
\end{lemma}
\begin{proof}
   As the proofs are similar, we consider only the case for $D_1(s, z; q_1,  \psi,\psi', \pi)$ here.  Notice that $D_1(s, z; q_1, \psi,\psi',\pi)$ are jointly multiplicative functions of $l,q_2$ for $i=1,2$ by Lemma \ref{lem:Gauss}. Moreover, as $\psi \neq \psi_0$, we may assume that $k$ is odd.  We then write $D_1(s, z; q_1, \psi,\psi', \pi)$  into an Euler product to see that
\begin{align}
\label{Dexp}
\begin{split}	
 &	D_1(s, z; q_1,  \psi, \psi', \pi)= \lz \sum_{l=0}^\infty\frac{ \psi'(2^{2l})}{2^{2ls}}\pz \prod_{(p,2)=1} \lz \sum_{k,l=0}^\infty\frac{ \psi(p^k)a_{\pi}(p^k)\psi'(p^{2l})G\lz \chi_{p^k}, q_1p^{2l} \pz }{p^{kz+2ls}}\pz.
\end{split}
\end{align}

For a fixed $p \neq 2$, we have
\begin{align*}
%%\label{Dgenest}	
\begin{split}
 \sum_{k,l=0}^\infty\frac{ \psi(p^k)a_{\pi}(p^k)\psi'(p^{2l})G\lz \chi_{p^k}, q_1p^{2l} \pz }{p^{kz+2ls}}
= \sum_{k=0}^\infty \frac{ \psi(p^k)a_{\pi}(p^k)G\lz \chi_{p^k}, q_1 \pz }{p^{kz}}  +
\sum_{k \geq 0, \ l \geq 1}\frac{ \psi(p^k)a_{\pi}(p^k)\psi'(p^{2l})G\lz \chi_{p^k}, q_1p^{2l} \pz }{p^{kz+2ls}}.
\end{split}
\end{align*}

  Further, as $q_1$ is square-free, we deduce from Lemma \ref{lem:Gauss} that
\begin{align*}
%%\label{Gest0}
\begin{split}
	|G( \chi_{p^k}, q_1p^{2l} )| \ll p^k, \quad G(\chi_{p^k}, q_1p^{2l})=0, \quad k \geq 2l+3.
\end{split}
\end{align*}

Now the above, together with \eqref{abound}, yeilds that if $\Re(s)>1/2$ and $\Re(z)>1$,
\begin{align}
\label{Dk1est}	
\begin{split}
\sum_{k \geq 0, l \geq 1}  \frac{ \psi(p^k)a_{\pi}(p^k)\psi'(p^{2l})G( \chi_{p^k}, q_1p^{2l} ) }{p^{kz+2ls}}
= & \sum_{l \geq 1}\frac{ \psi'(p^{2l})G( \chi_{1}, q_1p^{2l} ) }{p^{2ls}}+\sum_{l, k \geq 1}\frac{ \psi(p^k)a_{\pi}(p^k)\psi'(p^{2l})G( \chi_{p^k}, q_1p^{2l} ) }{p^{kz+2ls}} \\
\ll &\left| p^{-2s} + p^{-z+1}\sum_{l \geq 1}\frac{1}{p^{2ls}}(2l+3) \right|  \ll \left|  p^{-2s}+p^{-2s-z+1} \right|.
\end{split}
\end{align}

Furthermore, by Lemma \ref{lem:Gauss},
\begin{align}
\label{suml0}	
\begin{split}
 & \sum_{k=0}^\infty \frac{ \psi(p^k)a_{\pi}(p^k)G\lz \chi_{p^k}, q_1 \pz }{p^{kz}}
= \begin{cases}
\displaystyle{1+\frac{ \psi(p)a_{\pi}(p)\chi^{(q_1)}(p)}{p^{z-1/2}}}, & p \nmid q_1, \\ \\
\displaystyle{1-\frac {a_{\pi}(p^2)}{p^{2z-1}}}, & p |q_1.
\end{cases}
\end{split}
\end{align}

  We readily deduce from \eqref{Linverse}, \eqref{aexp}, \eqref{Dexp}--\eqref{suml0} that for $\Re(s)>1/2$, $\Re(z)>1$,  we may write
\begin{align}
\label{D1decomp}
\begin{split}
	D_1(s, z; q_1,  \psi, \psi', \pi)=L(z-\tfrac 12, \pi \otimes \chi^{(q_1)}\psi)^{-1}E_1(s, z; q_1, \psi, \psi', \pi),
\end{split}
\end{align}
 where in the region $\Re(z)>1+\varepsilon$, $\Re(s)>1+\varepsilon$, we have
\begin{align}
\label{D2est}	
\begin{split}
& E_1(s, z; q_1,  \psi, \psi', \pi) \ll \prod_{p|q_1}(1+O(p^{-2s}+p^{-2s-z+1/2}+p^{-z+1/2})) \ll q^{\varepsilon}_1.
\end{split}
\end{align}

  As the right-hand side of \eqref{D1decomp} is holomorphic in the region $\Re(z)>1$ and $\Re(s)>1$, the first assertion of the lemma follows.  We then apply \eqref{Linvgenupperbound} and \eqref{D2est} to obtain the estimate in \eqref{Diest}, thus completing the proof of the lemma.
\end{proof}
	
  It follows from \eqref{Cidef}, \eqref{C(s,w,z) as twisted C(s,w,z)} and the above lemma that the functions $C_i(s,z;\pi), i=1,2$ are now extended to the region
\begin{equation*}
%%\label{key}
		\{(s,z):\ \Re(s)>1,\ \Re(z)>1 \}.
\end{equation*}
	Using \eqref{A1A2}--\eqref{Functional equation in s} and the above enlargment, we extend $(s-1)A(s, z;\pi)$ to the region
\begin{equation*}
%%\label{key}
		S_4=\{(s, z):\ \Re(s+2z)>1, \ \Re(2z)>1, \ \Re(1-s)>1, \ \Re(s+z)>1 \}.
\end{equation*}
The condition $\Re(s+2z)>1$ is superseded by $\Re(2z)>1$ and $\Re(s+z)>1$ so that
\begin{equation*}
%%\label{key}
		S_4=\{(s, z):\ \Re(s+z)>1, \ \Re(s)<0 \},
\end{equation*}
as $\Re (2z) >1$ is also contained in the conditions appearing above. \newline

	We then get that the convex hull of $S_2$ and $S_4$ contains
\begin{align*}
%%\label{Final region of definition for A(s,w,z)}
		S_5=\{(s, z):\ \Re(z)> \tfrac12, \ \Re(s+z)>1, \ \Re(s+2z)>2 \}.
\end{align*}
Now Proposition~\ref{Bochner} again implies that $(s-1)A(s,z;\pi)$ converges absolutely in the region $S_5$.

\subsection{Bounding $A(s,z;\pi)$ in vertical strips}
\label{Section bound in vertical strips}
	
In this section, we estimate $|A(s,z;\pi)|$ in vertical strips, necessary in evaluating the integral in \eqref{Integral for all characters} below.  For the previously defined regions $S_j$, we set for any fixed $0<\delta <1/1000$,
\begin{equation*}
%%\label{Definition of S tilde}
		\tilde S_j=S_{j,\delta}\cap\{(s, z):\Re(s)>-5/2 , \; \Re (z) \leq 2\}, \quad \mbox{where} \quad S_{j,\delta}= \{ (s,z)+\delta (1,1) : (s, z) \in S_j \} .
\end{equation*}
Set
\begin{equation*}
%%\label{key}
		p(s)=s-1,
\end{equation*}
  so that $p(s)A(s,z;\pi)$ is an analytic function in the considered regions. \newline

   We first deduce, from the bound for $A(s,z;\pi)$ in \eqref{Abound}, that in $\tilde S_0$, under GRH,
\begin{align*}
%%\label{Lindelof}
\begin{split}
      |p(s)A(s,z;\pi)| \ll |(s+3)z|.
\end{split}
\end{align*}

   On the other hand, the expression for $A(s,z;\pi)$ in \eqref{Sum A(s,w,z) over n} and \eqref{Lchidbound} can be used to estimate $L( s, \chi^{(4k)})$ and yields that in the region $\tilde S_1$,
\begin{align*}
%%\label{Lindelof}
\begin{split}
      |p(s)A(s,z;\pi)| \ll |(s+3)^8 z|.
\end{split}
\end{align*}

Proposition \ref{Extending inequalities} gives that in the convex hull of $\tilde S_2$ of $\tilde S_0$ and $\tilde S_1$,
\begin{equation}
\label{AboundS2}
		|p(s)A(s,z;\pi)|\ll |(s+3)^8 z| .
\end{equation}

   Moreover, by the convexity bound contained in \eqref{Lchidbound} for $\zeta(s)$ (corresponding to the case $d=1$), we see that in the region $\tilde S_3$, we have under GRH,
\begin{align}
\label{A1bound}
		|A_1(s,z;\pi)| \ll  |s+3|^{4}.
\end{align}
We note here that one can easily check, following the computations in the sequel, subconvexity bounds do not yield any improvement to our main result. \newline

   Also, by \eqref{Cexp}--\eqref{Didef} and Lemma \ref{Estimate For D(w,t)}, we have
\begin{equation}
\label{Csbound}
		|C(s,z;\pi)|\ll |z| \quad \mbox{for} \quad \Re(s)>1+\varepsilon , \; \Re(z)>1+\varepsilon.
\end{equation}

  Now, applying \eqref{Stirlingratio1}, \eqref{A1A2}, the functional equation \eqref{Functional equation in s} and the bounds in \eqref{A1bound} and \eqref{Csbound}, we obtain by Proposition \ref{Extending inequalities} that, in the region $\tilde S_4$,
\begin{equation}
\label{AboundS3}
		|p(s)A(s,z;\pi)|\ll (|z(s+3)^5|.
\end{equation}

   We deduce from \eqref{AboundS2}, \eqref{AboundS3} and apply Proposition \ref{Extending inequalities} again to conclude that in $\tilde S_5$,
\begin{equation}
\label{AboundS4}
		|p(s)A(s,z;\pi)|\ll |z(s+3)^8|.
\end{equation}
	
\subsection{Completion of proof}

   Recall that the Mellin transform $\hat{f}$ of any function $f$ is defined as
\begin{align*}
     \widehat{f}(s) =\int\limits^{\infty}_0f(t)t^s\frac {\dif t}{t}.
\end{align*}

   We now apply the Mellin inversion to obtain that
\begin{equation}
\label{Integral for all characters}
		\sum_{\substack{(n,2)=1}}\frac{1}{L^{(2)}(\tfrac{1}{2}+\alpha, \pi \otimes \chi_{n})}w \bfrac {n}X=\frac1{2\pi i}\int\limits_{(2)}A\lz s,\tfrac12+\alpha; \pi \pz X^s\widehat w(s) \dif s,
\end{equation}
  where $A(s, z;\pi)$ is defined in \eqref{Aswzexp}. \newline

    We then shift the integral in \eqref{Integral for all characters} to $\Re(s)=1-2\alpha+\varepsilon$. We encounter a simple pole at $s=1$ in this process with the corresponding residue given in \eqref{residues=1} with $z=1/2+\alpha$ there. This yeilds the main terms in \eqref{Asymptotic for ratios of all characters}. \newline

We shall use \eqref{AboundS4} to estimate the integral on the new line.  On this line, we have
\begin{equation}
\label{Bound in vertical strips}
		|A(s,z;\pi)|\ll |z(s+3)^8|.
\end{equation}
    Moreover, integration by parts implies, mindful of the compact support of $w$ in $(0,\infty)$, that for any integer $A \geq 0$,
\begin{align*}
%%\label{whatbound}
 \widehat w(s)  \ll  \frac{1}{(1+|s|)^{A}}.
\end{align*}
We apply the above estimation by setting $A=20$ together with \eqref{Bound in vertical strips} to get that the integral on the new line can be absorbed into the $O$-term in \eqref{Asymptotic for ratios of all characters}.  This completes the proof of Theorem \ref{Theorem for all characters}.

\vspace*{.5cm}

\noindent{\bf Acknowledgments.}   P. G. is supported in part by NSFC grant 11871082 and L. Z. by the FRG grant PS43707 at the University of New South Wales.

\bibliography{biblio}
\bibliographystyle{amsxport}

\end{document}